\documentclass[11pt]{article}

\usepackage[T1]{fontenc}
\usepackage[utf8]{inputenc}
\usepackage{lmodern}

\usepackage[a4paper,margin=1in]{geometry}
\usepackage{amsmath,amssymb,amsthm,mathtools}
\usepackage{hyperref}
\usepackage{enumitem}

\hypersetup{
  hidelinks,
  pdftitle={An Explicit Cubic Ramanujan--Sato Series for 1/pi on Gamma0(2)+ at D=-163},
  pdfauthor={Vedran Mendjusic},
  pdfsubject={Ramanujan--Sato series for 1/pi},
  pdfkeywords={Ramanujan--Sato, modular forms, complex multiplication, Gamma0(2), pi},
  pdfcreator={LaTeX},
  pdfproducer={pdfTeX}
}

\title{An Explicit Cubic Ramanujan--Sato Series for \texorpdfstring{\(1/\pi\)}{1/pi} on \texorpdfstring{\(\Gamma_0(2)^+\)}{Gamma0(2)+} at \texorpdfstring{\(D=-163\)}{D=-163}}
\author{Vedran Menđušić}
\date{June 27, 2026}

\newtheorem{theorem}{Theorem}
\newtheorem{proposition}{Proposition}
\newtheorem{lemma}{Lemma}

\DeclareMathOperator{\Norm}{Norm}

\newcommand{\SL}{\mathrm{SL}}

\begin{document}

\maketitle

\begin{abstract}
This note records an explicit cubic Ramanujan--Sato type formula for
\(1/\pi\) attached to the Fricke group \(\Gamma_0(2)^+\) and the Heegner
discriminant \(D=-163\).  The hypergeometric kernel is
\[
{}_3F_2\!\left(\frac14,\frac12,\frac34;1,1;X\right)
=
\sum_{n\ge0}\frac{(4n)!}{256^n(n!)^4}X^n.
\]
The construction uses the Hauptmodul \(H=w+64/w\), the modular parameter
\(X=256/H^2\), the singular value \(j=-640320^3\), and the Masser--Milla
almost-holomorphic CM evaluation specified in Section~7.  The resulting
parameter has absolute value about \(10^{-15}\), so the series has a rapid
per-term contraction, while its coefficients remain cubic algebraic numbers.
\end{abstract}



\section{Introduction and relation to previous work}

Ramanujan--Sato series for \(1/\pi\) are a well-developed part of the theory
of modular forms and complex multiplication, going back to Ramanujan's
formulas \cite{Ramanujan} and their modern interpretations through modular
parametrizations \cite{BorweinAGM}. Chan--Cooper type constructions organize such series by levels associated
with congruence subgroups, and later work has produced systematic catalogues
of rational and quadratic examples; see, for example,
\cite{HuberSchultzYe,Level17,WeiTaoGuo,CampbellCooperYe}.

The present formula lies just outside those catalogued cases.  The Heegner
discriminant \(D=-163\) has class number one \cite{Cox}, so the classical singular
modulus
\[
j\!\left(\frac{1+\sqrt{-163}}2\right)=-640320^3
\]
is rational and leads in the usual level-one setting to the Chudnovsky-type
series \cite{Chudnovsky,Campbell163}.  Here the same CM point is evaluated
through the level-two Hauptmodul
\[
t(\tau)=\left(\frac{\eta(\tau)}{\eta(2\tau)}\right)^{24}.
\]
Since
\[
j=\frac{(t+256)^3}{t^2},
\]
the value \(t(\tau_{163})\) is cubic over \(\mathbb Q\) even though
\(j(\tau_{163})\in\mathbb Q\).  Cubic Ramanujan-type series are known in other
senses, including the class-number-three examples of Borwein and Borwein and
the signature-three constructions developed by Chan and Liaw
\cite{BorweinClassNumber,ChanLiaw}.  The present cubic field has a different
origin: it comes from pulling the rational singular modulus
\(j(\tau_{163})=-640320^3\) through the degree-three covering
\(X_0(2)\to X(1)\), rather than from a class-number-three CM field or the
signature-three theory.

The relation with the Chudnovsky formula should be stated carefully.  Both
identities use the same Heegner discriminant \(D=-163\) and the same level-one
singular modulus \(j(\tau_{163})=-640320^3\).  The Chudnovsky series is the
level-one rational-coefficient specialization, whereas the identity here pulls
the same CM point through the degree-three covering
\(X_0(2)\to X(1)\) and then uses the Fricke-invariant parameter
\(X=256/H^2\).  Thus the two formulas are CM-related, but the present
\({}_3F_2(1/4,1/2,3/4;1,1;X)\) identity is not presented as a mere numerical
refitting of the Chudnovsky coefficients.  Campbell's work on the Heegner number
163 and related extensions of the Chudnovsky algorithm provides the closest
level-one comparison point \cite{Campbell163,CampbellExtension}.

The speed is noteworthy but not algorithmically free.  The contraction is
slightly stronger per term than in the classical Chudnovsky formula at
\(D=-163\), but the coefficients here are cubic algebraic numbers rather than
rational integers.  The formula should therefore be viewed primarily as an
explicit CM/Ramanujan--Sato identity, not as a practical replacement for the
optimized Chudnovsky method.

To the author's knowledge, the resulting explicit \(\Gamma_0(2)^+\),
\(D=-163\) specialization of
\({}_3F_2(1/4,1/2,3/4;1,1;X)\), with the cubic constants denoted below by
\(\xi\) and \(\alpha\), has not appeared in the indexed literature.

\section{Modular setup}

Let
\[
q=e^{2\pi i\tau},
\qquad
t(\tau)=\left(\frac{\eta(\tau)}{\eta(2\tau)}\right)^{24}.
\]
The classical relation between \(t\) and the absolute modular invariant is
\begin{equation}
j(\tau)=\frac{(t(\tau)+256)^3}{t(\tau)^2}.
\label{eq:j-t}
\end{equation}

Define
\[
w(\tau)=\left(\sqrt2\,\frac{\eta(2\tau)}{\eta(\tau)}\right)^{12}.
\]
Then
\[
w(\tau)^2=\frac{4096}{t(\tau)}.
\]
Under the Fricke involution
\[
W_2:\tau\mapsto -\frac1{2\tau},
\]
one has
\[
w(W_2\tau)=\frac{64}{w(\tau)}.
\]
Hence the Fricke-invariant Hauptmodul is
\[
H(\tau)=w(\tau)+\frac{64}{w(\tau)}.
\]

Put
\[
S(\tau)=H(\tau)^2-128.
\]
Then
\begin{equation}
S(\tau)=t(\tau)+\frac{4096}{t(\tau)}.
\label{eq:S-t}
\end{equation}
Eliminating \(t\) between \eqref{eq:j-t} and \eqref{eq:S-t} gives
\begin{equation}
j^2-(S^2+49S-6656)j+(S+272)^3=0.
\label{eq:j-S}
\end{equation}

\section{The hypergeometric parameter}

Let
\[
F(X)=
{}_3F_2\left(\frac14,\frac12,\frac34;1,1;X\right).
\]
Then
\[
F(X)=
\sum_{n=0}^{\infty}
\frac{(1/4)_n(1/2)_n(3/4)_n}{(n!)^3}X^n
=
\sum_{n=0}^{\infty}
\frac{(4n)!}{256^n(n!)^4}X^n.
\]

For the \(\Gamma_0(2)^+\) branch set
\begin{equation}
X(\tau)=\frac{256}{H(\tau)^2}.
\label{eq:X-H}
\end{equation}
Since
\[
H^2=
\left(w+\frac{64}{w}\right)^2
=
\frac{(t+64)^2}{t},
\]
one also has
\begin{equation}
X(\tau)=\frac{256t(\tau)}{(t(\tau)+64)^2}.
\label{eq:X-t}
\end{equation}
Consequently,
\begin{equation}
1-X=
\frac{(t-64)^2}{(t+64)^2},
\qquad
\sqrt{1-X}=\frac{t-64}{t+64}
\label{eq:sqrt1X}
\end{equation}
after fixing the local branch at the cusp.

\section{Modular verification of the hypergeometric identity}

Define
\[
G(\tau)=2E_2(2\tau)-E_2(\tau).
\]
The \(q\)-expansion starts with \(G=1+24q+\cdots\), and
\[
q\frac{d}{dq}\log t
=
E_2(\tau)-2E_2(2\tau)
=
-G(\tau).
\]

\begin{proposition}
For \(X=256t/(t+64)^2\),
\begin{equation}
F(X(\tau))=G(\tau).
\label{eq:F-G}
\end{equation}
\end{proposition}

\begin{proof}
The function
\[
F(X)=
{}_3F_2\left(\frac14,\frac12,\frac34;1,1;X\right)
\]
is the normalized local solution at \(X=0\) of
\begin{equation}
\theta_X^3F
-
X\left(\theta_X+\frac14\right)
\left(\theta_X+\frac12\right)
\left(\theta_X+\frac34\right)F=0,
\qquad
\theta_X=X\frac{d}{dX}.
\label{eq:3f2-ode}
\end{equation}
We show that \(G\), considered as a local function of \(X\), satisfies the same
equation.  Put
\[
u=\frac{E_2}{G},\qquad
s=\sqrt{1-X}=\frac{t-64}{t+64},
\qquad
\Theta=q\frac{d}{dq}.
\]
From \(\Theta\log X=Gs\), one has
\[
\theta_X=\frac{1}{Gs}\Theta.
\]
Using Ramanujan's differential relations and the rational modular ratios
\[
R_4=\frac{E_4}{G^2}=\frac{t+256}{t+64},
\qquad
R_{42}=\frac{E_4(2\tau)}{G^2}=
\frac{t^2+80t+1024}{(t+64)^2},
\]
proved below, gives the closed differential system
\[
\Theta t=-Gt,
\]
\[
\Theta G=
\frac{G^2}{12}\left(1+2u+R_4-4R_{42}\right),
\]
and
\[
\Theta u
=
-\frac{G\left(tu^2-2tu+t+64u^2+256u+256\right)}
{12(t+64)}.
\]
Applying \(\theta_X=(Gs)^{-1}\Theta\) repeatedly gives
\[
\theta_X^rG=GP_r(t,u),\qquad r=1,2,3,
\]
with \(P_r\in\mathbb Q(t,u)\).  Substitution into the left hand side of
\eqref{eq:3f2-ode} reduces the Picard--Fuchs residual to
\[
G\left[
P_3-
X\left(P_3+\frac32P_2+\frac{11}{16}P_1+\frac{3}{32}\right)
\right],
\qquad
X=\frac{256t}{(t+64)^2}.
\]
A direct simplification in the rational function field \(\mathbb Q(t,u)\)
gives zero.  Hence \(G\) satisfies the pulled-back hypergeometric equation.

Finally, \(X=256q+O(q^2)\), so \(X\) is a local coordinate at the cusp, and
\(G=1+O(q)\).  Thus \(G(q(X))\) is a holomorphic local solution at \(X=0\)
with value \(1\).  By uniqueness of the normalized holomorphic local solution
of \eqref{eq:3f2-ode}, \(G(\tau)=F(X(\tau))\).
\end{proof}

Using \(X=256t/(t+64)^2\), one obtains
\begin{equation}
q\frac{d}{dq}\log X
=
G(\tau)\frac{t-64}{t+64}
=
G(\tau)\sqrt{1-X}.
\label{eq:dlogX}
\end{equation}
Thus
\[
q\frac{d}{dq}F(X(\tau))
=
G(\tau)\sqrt{1-X(\tau)}\,\theta F(X(\tau)),
\qquad
\theta=X\frac{d}{dX}.
\]

\section{Rational modular ratios}

The derivation of \(A\) uses four rational modular ratios:
\begin{align}
\frac{E_4(\tau)}{G(\tau)^2}
&=
\frac{t+256}{t+64},
\label{eq:R4}
\\
\frac{E_4(2\tau)}{G(\tau)^2}
&=
\frac{t^2+80t+1024}{(t+64)^2},
\label{eq:R42}
\\
\frac{E_6(\tau)}{G(\tau)^3}
&=
\frac{t-512}{t+64},
\label{eq:R6}
\\
\frac{\Delta(\tau)}{G(\tau)^6}
&=
\frac{t^2}{(t+64)^3}.
\label{eq:RDelta}
\end{align}

\begin{lemma}[Finite divisor check]
The identities \ref{eq:R4}--\ref{eq:RDelta} are identities of modular
functions on \(X_0(2)\).
\end{lemma}

\begin{proof}
After clearing denominators, put
\begin{align*}
D_1&=(t+64)E_4-(t+256)G^2,\\
D_2&=(t+64)^2E_4(2\tau)-(t^2+80t+1024)G^2,\\
D_3&=(t+64)E_6-(t-512)G^3,\\
D_4&=(t+64)^3\Delta-t^2G^6.
\end{align*}
A priori, each \(D_i\) is a meromorphic modular form on \(\Gamma_0(2)\); the only
possible pole comes from powers of the Hauptmodul \(t=q^{-1}+O(1)\) at the cusp
\(\infty\).  There are no denominators involving \(t\) in the cleared expressions.
The formal \(q\)-expansion certificate verifies that the principal part at
\(\infty\) cancels and that
\[
D_1=O(q^{14}),\qquad
D_2=O(q^{14}),\qquad
D_3=O(q^{14}),\qquad
D_4=O(q^{14}).
\]
The relevant bounds are as follows.
\begin{center}
\begin{tabular}{c|c|c|c}
cleared difference & weight & a priori pole order at \(\infty\) & valence bound \\
\hline
\(D_1\) & \(4\) & \(1\) & \(1\) \\
\(D_2\) & \(4\) & \(2\) & \(1\) \\
\(D_3\) & \(6\) & \(1\) & \(3/2\) \\
\(D_4\) & \(12\) & \(2\) & \(3\)
\end{tabular}
\end{center}
Indeed,
\([\SL_2(\mathbb Z):\Gamma_0(2)]=3\), so a nonzero holomorphic modular form of
weight \(k\) on \(\Gamma_0(2)\) has total divisor degree \(k/4\).  Once the
\(q\)-certificate has removed the possible principal part at the cusp, each
\(D_i\) is holomorphic.  Its vanishing order at \(\infty\) is at least \(14\), which
is larger than \(k/4\) in all four cases.  Hence every \(D_i\) is identically zero.
\end{proof}

\section{The CM point \texorpdfstring{\(D=-163\)}{D=-163}}

Let
\[
\tau_{163}=\frac{1+\sqrt{-163}}2.
\]
Then
\[
j(\tau_{163})=-640320^3.
\]
Thus \(t=t(\tau_{163})\) satisfies
\[
(t+256)^3+640320^3t^2=0,
\]
i.e.
\begin{equation}
P_t(t)=
t^3
+
262537412640768768t^2
+
196608t
+
16777216
=0.
\label{eq:Pt}
\end{equation}
Let
\[
K=\mathbb Q(t),\qquad P_t(t)=0,
\]
be the corresponding cubic field.

Define
\begin{equation}
\xi=X(\tau_{163})
=
\frac{256t}{(t+64)^2}.
\label{eq:xi-def}
\end{equation}
Eliminating \(t\) from \eqref{eq:Pt} and \eqref{eq:xi-def} gives
\begin{equation}
\begin{aligned}
0={}&
16827610604518993301932059648729\,\xi^3\\
&-
3396577776039932112\,\xi^2
+
4200598602252294912\,\xi
+
4096.
\end{aligned}
\label{eq:Pxi}
\end{equation}
The chosen branch is the real small root
\[
\xi\approx
-9.7509911987395938485584703467480\times10^{-16}.
\]
This decimal value is only illustrative; throughout the formula below,
\(\xi\) denotes the specified real root of the cubic polynomial
\eqref{eq:Pxi}.

\section{The coefficients \texorpdfstring{\(A\) and \(B\)}{A and B}}

At
\[
\tau_{163}=\frac{1+\sqrt{-163}}2
\]
one has
\[
\operatorname{Im}(\tau_{163})=\frac{\sqrt{163}}2.
\]
The almost-holomorphic Eisenstein series
\[
E_2^*(\tau)=E_2(\tau)-\frac{3}{\pi\,\operatorname{Im}\tau}
\]
therefore gives
\[
E_2(\tau_{163})
=
E_2^*(\tau_{163})+\frac{6}{\pi\sqrt{163}}.
\]
The non-holomorphic corrections also cancel in the combination
\[
2E_2^*(2\tau)-E_2^*(\tau)=2E_2(2\tau)-E_2(\tau)=G(\tau),
\]
because \(\operatorname{Im}(2\tau)=2\operatorname{Im}(\tau)\).  Thus the
almost-holomorphic input used below is compatible with the same modular form
\(G\) used in the \(q\)-expansion verification.

The Ramanujan--Sato linear form used below is the following specialization of
the standard CM construction.

\begin{lemma}[Linear form]
Let \(\tau=(1+\sqrt{-d})/2\) be on the branch above, and set
\(X=X(\tau)\).  With \(F\), \(G\), and \(t\) as above,
\begin{equation}
\frac1\pi
=
\sqrt d\left(\mathcal A(\tau)F(X)+\sqrt{1-X}\,\theta F(X)\right),
\label{eq:linear-form}
\end{equation}
where
\begin{equation}
\mathcal A(\tau)=\frac1{12}\left(
4\frac{E_4(2\tau)}{G(\tau)^2}
-\frac{E_4(\tau)}{G(\tau)^2}
-1
-2\frac{E_2^*(\tau)}{G(\tau)}
\right).
\label{eq:A-general}
\end{equation}
\end{lemma}

\begin{proof}
From \(F(X(\tau))=G(\tau)\) and \eqref{eq:dlogX},
\[
\frac{q\,dG/dq}{G}
=
\frac{q\,dF(X(\tau))/dq}{G}
=
\frac{G\sqrt{1-X}\,\theta F(X)}{G}
=
\sqrt{1-X}\,\theta F(X).
\]
Ramanujan's differential identities give
\[
q\frac{dE_2}{dq}=\frac{E_2^2-E_4}{12},
\qquad
q\frac{dE_2(2\tau)}{dq}=\frac{E_2(2\tau)^2-E_4(2\tau)}6.
\]
Since \(G=2E_2(2\tau)-E_2(\tau)\),
\[
\begin{aligned}
q\frac{dG}{dq}
&=\frac{4E_2(2\tau)^2-E_2(\tau)^2-4E_4(2\tau)+E_4(\tau)}{12} \\
&=\frac{G^2+2E_2G+E_4-4E_4(2\tau)}{12}.
\end{aligned}
\]
Therefore
\[
\frac{q\,dG/dq}{G}
=\frac1{12}\left(
G+2E_2+\frac{E_4}{G}-4\frac{E_4(2\tau)}{G}
\right).
\]
Also, from the definition of \(\mathcal A\),
\[
\mathcal A(\tau)G
=\frac1{12}\left(
4\frac{E_4(2\tau)}{G}-\frac{E_4}{G}-G-2E_2^*
\right).
\]
Adding the last two displayed formulas gives the cancellation
\[
\mathcal A(\tau)G+\frac{q\,dG/dq}{G}
=\frac{E_2-E_2^*}{6}.
\]
Since
\[
E_2-E_2^*=\frac{3}{\pi\operatorname{Im}\tau},
\]
we obtain
\[
\mathcal A(\tau)F(X)+\sqrt{1-X}\,\theta F(X)
=\frac{1}{2\pi\operatorname{Im}\tau}.
\]
For \(\tau=(1+\sqrt{-d})/2\), this is \(1/(\pi\sqrt d)\).  Multiplication by
\(\sqrt d\) proves \eqref{eq:linear-form}.
\end{proof}

The coefficient multiplying \(\theta F\) in \eqref{eq:linear-form} at
\(d=163\) is
\begin{equation}
B_{-163}=\sqrt{163}\sqrt{1-\xi}.
\label{eq:B}
\end{equation}

For \(A\), use Masser's almost-holomorphic modular functions
\[
\chi^*(\tau)=\frac{E_2^*(\tau)E_4(\tau)E_6(\tau)}{\Delta(\tau)},
\qquad
\psi(\tau)=\frac{E_2^*(\tau)E_4(\tau)}{E_6(\tau)}.
\]
With the normalizations
\[
\Delta=\frac{E_4^3-E_6^2}{1728},
\qquad
j=\frac{E_4^3}{\Delta},
\]
one has
\begin{equation}
\chi^*(\tau)=\psi(\tau)(j(\tau)-1728).
\label{eq:chi-psi}
\end{equation}
Indeed, \(E_6^2=E_4^3-1728\Delta\), so \(E_6^2/\Delta=j-1728\).
Masser's theory of almost-holomorphic CM values, in the normalization used by
Milla, gives
\[
\psi(\tau_{163})=\frac{77265280}{90856689}
\qquad \cite{Masser,Milla}.
\]
Together with \(j(\tau_{163})=-640320^3\), \eqref{eq:chi-psi} gives
\[
\begin{aligned}
\chi^*(\tau_{163})
&=\frac{77265280}{90856689}\left(-640320^3-1728\right)\\
&=-223263987730882560,
\end{aligned}
\]
because \(-640320^3-1728=-2889577152\cdot90856689\).  Thus
\begin{equation}
\chi^*(\tau_{163})=-223263987730882560.
\label{eq:chi}
\end{equation}
This CM value is an external input from Masser--Milla theory, not a coefficient
fitted from the final \(1/\pi\) identity.

Let \(\chi\) denote the value in \eqref{eq:chi}. Put
\[
u=\frac{E_2^*(\tau_{163})}{G(\tau_{163})}.
\]
By \eqref{eq:R4}, \eqref{eq:R6}, and \eqref{eq:RDelta},
\[
u=
\chi
\frac{t^2}{(t+64)(t+256)(t-512)}.
\]

Writing
\[
A_{-163}=\sqrt{163}\,\alpha,
\]
equations \eqref{eq:A-general}, \eqref{eq:R4}, \eqref{eq:R42}, and the above
expression for \(u\) give
\begin{equation}
\begin{aligned}
\alpha=
\frac{1}{12}\Bigg(&
4\frac{t^2+80t+1024}{(t+64)^2}
-
\frac{t+256}{t+64}
-
1\\
&-
2\chi
\frac{t^2}{(t+64)(t+256)(t-512)}
\Bigg).
\end{aligned}
\label{eq:alpha-def}
\end{equation}
Thus \(\alpha\in K\). Taking the norm
\[
\Norm_{K/\mathbb Q}(Y-\alpha)
\]
gives
\begin{equation}
668649972819460401Y^3
+
50012252033677839Y^2
+
1467Y
-
41450311432931=0.
\label{eq:Palpha}
\end{equation}
The chosen branch is the positive real root
\[
\alpha
\approx
0.02493195446879323036925498953687225.
\]

\section{Formula}

\begin{theorem}
Let \(\xi\) be the real small root
of \eqref{eq:Pxi}, and let \(\alpha\) be the positive real root of
\eqref{eq:Palpha}. Then
\begin{equation}
\boxed{
\frac1\pi
=
\sqrt{163}
\sum_{n=0}^{\infty}
\frac{(4n)!}{256^n(n!)^4}
\xi^n
\Bigl(
\alpha+\sqrt{1-\xi}\,n
\Bigr).
}
\label{eq:main}
\end{equation}
\end{theorem}

In \eqref{eq:main}, the coefficient of \(n\) inside the parentheses is
\(\sqrt{1-\xi}\); after multiplication by the prefactor \(\sqrt{163}\), this
is exactly the coefficient \(B_{-163}\) from \eqref{eq:B}.

Since
\[
|\xi|\approx9.7509911987\times10^{-16},
\]
the series yields approximately \(15.01\) decimal digits per term.

\appendix

\section{Algebraic provenance and verification}
\label{app:provenance}

This appendix records the status of the auxiliary CM evaluation and the exact
algebraic checks used in the construction.  The final identity for \(1/\pi\)
is not used to determine the constants appearing in Theorem~1.

\subsection{CM input}

The value \eqref{eq:chi} is used with the normalization
\[
\chi^*(\tau)
=
\frac{E_2^*(\tau)E_4(\tau)E_6(\tau)}{\Delta(\tau)},
\qquad
E_2^*(\tau)=E_2(\tau)-\frac{3}{\pi\operatorname{Im}\tau}.
\]
It is derived in Section~7 from the Masser--Milla CM value of
\(\psi=E_2^*E_4/E_6\) and the identity \(\chi^*=\psi(j-1728)\).  The ancillary
verification scripts do not determine this value from the final series; they
verify only the algebraic consequences after the independent CM input has been
specified.

\medskip
\noindent\textbf{Summary.}  The constants \(t\), \(\xi\), and \(\alpha\) are defined
algebraically before the final \(1/\pi\) evaluation.  The final formula is not
used to fit any of these constants.

\subsection{Exact algebraic derivation of the cubic for \texorpdfstring{\(\alpha\)}{alpha}}

With \(\chi\) denoting the value in \eqref{eq:chi}, the coefficient \(\alpha\)
is the explicit element of \(K=\mathbb Q(t)\) defined by
\[
\alpha
=
\frac{1}{12}
\left[
4\frac{t^2+80t+1024}{(t+64)^2}
-
\frac{t+256}{t+64}
-
1
-
2\chi
\frac{t^2}{(t+64)(t+256)(t-512)}
\right].
\]
Here \(t=t(\tau_{163})\) satisfies
\[
P_t(t)
=
t^3+262537412640768768t^2+196608t+16777216=0,
\]
which follows from \(j(\tau_{163})=-640320^3\) and
\(j=(t+256)^3/t^2\).

Writing
\[
\alpha=\frac{N_\alpha(t)}{D_\alpha(t)},
\qquad N_\alpha,D_\alpha\in\mathbb Q[t],
\]
the exact symbolic check gives
\[
\operatorname{Res}_t
\left(
P_t(t),
Y\,D_\alpha(t)-N_\alpha(t)
\right)
=
c\,P_\alpha(Y),
\qquad
c\in\mathbb Q^\times,
\]
with zero remainder, where
\[
P_\alpha(Y)
=
668649972819460401Y^3
+
50012252033677839Y^2
+
1467Y
-
41450311432931.
\]
Equivalently,
\[
P_\alpha(Y)=\operatorname{Norm}_{K/\mathbb Q}(Y-\alpha)
\]
up to primitive normalization.  Thus the cubic for \(\alpha\) is the norm
polynomial of an explicitly defined element of \(\mathbb Q(t)\).

\subsection{Finite \texorpdfstring{\(q\)}{q}-expansion verification}

The modular identities used in the derivation are checked by formal
\(q\)-expansion in the ancillary verification scripts.  The proof mechanism is
the finite divisor check in Section~5: after clearing denominators, the only
possible pole comes from the Hauptmodul \(t=q^{-1}+O(1)\) at the cusp.  The
\(q\)-certificate verifies that this principal part cancels and that the cleared
differences vanish to order at least \(14\), which is beyond the valence bounds
\(1,1,3/2,3\) for weights \(4,4,6,12\) on \(\Gamma_0(2)\).

\subsection{Reproducibility}

The ancillary files contain verification scripts which check the algebraic
relations for \(t\), \(\xi\), and \(\alpha\), the resultant computation, the
numerical convergence of the displayed series, and the \(q\)-expansion
identities used in Sections 4 and 5.

\end{document}